\documentclass[11pt]{amsart}

\usepackage{amssymb}
\usepackage{amsmath}
\usepackage{amsthm, times, fullpage, float}
\usepackage{caption}
\usepackage{pstricks}
\usepackage{color}

\usepackage{palatino}

\newtheorem{theo}{Theorem}[section]

\theoremstyle{definition}

\newtheorem{example}[theo]{Example}
\newtheorem{definition}[theo]{Definition}

\theoremstyle{plain}
\newtheorem{lemma}[theo]{Lemma}
\newtheorem{theorem}[theo]{Theorem}

\theoremstyle{definition}

\newtheorem{remark}[theo]{Remark}

\newcommand{\beq}{\begin{equation}}
\newcommand{\eeq}{\end{equation}}
\renewcommand{\a}{\alpha}
\renewcommand{\b}{\beta}

\renewcommand{\d}{\delta}
\newcommand{\e}{\epsilon}

\renewcommand{\l}{\lambda}

\renewcommand{\o}{\omega}

\newcommand{\bC}{\mathbb{C}}

\newcommand{\bR}{\mathbb{R}}

\newcommand{\gc}{\mathfrak{c}}

\renewcommand{\gg}{\mathfrak{g}}
\newcommand{\gh}{\mathfrak{h}}

\newcommand{\gm}{\mathfrak{m}}

\newcommand{\gq}{\mathfrak{q}}

\newcommand{\gt}{\mathfrak{t}}

\newcommand{\so}{\mathfrak{so}}
\newcommand{\su}{\mathfrak{su}}

\newcommand\SU{\mathrm{SU}}
\newcommand\U{\mathrm{U}}

\renewcommand\sp{\mathfrak{sp}}

\renewcommand{\square}{\kern1pt\vbox
{\hrule height 0.6pt\hbox{\vrule width 0.6pt\hskip 3pt
\vbox{\vskip 6pt}\hskip 3pt\vrule width 0.6pt}\hrule height0.6pt}\kern1pt}

\DeclareMathOperator\End{End\;}

\DeclareMathOperator\ad{ad}

\newcommand{\be}{\begin{equation}}
\newcommand{\ee}{\end{equation}}

\def\<#1,#2>{\langle\,#1,\,#2\,\rangle}
\newcommand{\arr}{\begin{array}{rlll}}
\newcommand{\ea}{\end{array}}
\newcommand{\bea}{\begin{eqnarray}}
\newcommand{\eea}{\end{eqnarray}}
\newcommand{\bean}{\begin{eqnarray*}}
\newcommand{\eean}{\end{eqnarray*}}

\def\sideremark#1{\ifvmode\leavevmode\fi\vadjust{
\vbox to0pt{\hbox to 0pt{\hskip\hsize\hskip1em
\vbox{\hsize3cm\tiny\raggedright\pretolerance10000
\noindent #1\hfill}\hss}\vbox to8pt{\vfil}\vss}}}
\newcounter{ssig}
\newcounter{ttig}
\title[On the first eigenvalue of invariant K\"ahler metrics] {On the first eigenvalue of invariant K\"ahler metrics}

\author{Francesco Panelli and Fabio Podest\`a }
\address{Dipartimento di Matematica e Informatica "Ulisse Dini", Universit\`a di Firenze, V.le Morgagni 67/A, 50100 Firenze, Italy}
\email{podesta@unifi.it,\ francescopanelli@virgilio.it}


\begin{document}
\begin{abstract} Given a simply connected compact generalized flag manifold $M$ together with its invariant K\"ahler Einstein metric $\bar g$, we investigate the functional given by the first eigenvalue of the Hodge Laplacian on $C^\infty(M)$ restricted to the space of invariant K\"ahler metrics. We give sufficient and necessary conditions so that the metric $\bar g$ is a critical point for this functional. Moreover we prove that when $M$ is a full flag manifold, the metric $\bar g$ is critical if and only if $M= \SU(3)/T^2$ and in this case $\bar g$ is a maximum.
\end{abstract}

\maketitle
\section{Introduction}
The first eigenvalue $\lambda_1(g)$ of the Laplacian $\Delta_g$ acting on smooth functions on a Riemannian manifold $(M,g)$ is a very important, widely investigated geometrical object (see e.g.\cite{BGM}). Given a compact manifold $M$ and a class of Riemannian metrics $\mathcal R$ on $M$, it is a natural and interesting problem to investigate the extrema and boundedness of the functional $\mathcal R \ni g\mapsto \lambda_1(g)$.  It is known that the functional $\lambda_1$ is unbounded when $\mathcal R$ is the space of metrics of fixed volume and $\dim M\geq 3$ (see \cite {CD}). In the K\"ahler setting, if a compact K\"ahler manifold $(M,g)$ admits a holomorphic isometric embedding into some complex projective space, then the functional $\lambda_1$ is bounded on the space of K\"ahler metrics whose K\"ahler form belongs to the K\"ahler class $[\omega_g]$ (\cite{BLY}). Similarly, when $(M,g)$ is Hodge the functional $\lambda_1$ is bounded on the space of all K\"ahler metrics $g'$ which are K\"ahler w.r.t. a suitable complex structure $J'$ on $M$ and whose K\"ahler form is $\omega_g$ (see \cite{Po}).\par
More recently Biliotti and Ghigi (\cite{BG}) showed that on a compact Hermitian symmetric space $N$ of type ABCD endowed with the standard invariant metric $g_o$ with $Ric(g_o) = g_o$ we have that $\lambda_1(g)\leq 2 = \lambda_1(g_o)$, where $g$ is any K\"ahler metric with K\"ahler class $[\omega_g] \in 2\pi c_1(M)$. When $N$ is irreducible, this is equivalent to saying that the functional $g\mapsto \lambda_1(g)$ restricted to the space of K\"ahler metrics with fixed volume attains its maximum at the K\"ahler Einstein metric. This last reformulation is consistent with the fact that on a compact  homogeneous Riemannian manifold $(M,g)$ with irreducible isotropy representation the invariant metric $g$ is {\it extremal\/} for the functional $\lambda_1$ on the space of all Riemannian metrics $g'$ on $M$ with $vol(M,g')=vol(M,g)$ (see \cite{SI}). An appropriate notion of extremality for the functional $\lambda_1$ has been defined in \cite{SI} and in this paper we will investigate
 it in the case of a compact simply connected homogeneous K\"ahler manifold $(M,g,J)$  when $\mathcal R$ is given by the set of all K\"ahler invariant metrics of fixed volume. \par
When $(M,g,J)$ is a compact connected K\"ahler manifold  which is homogeneous under the action of a compact semisimple Lie group $G$, it is well known that $M$ admits a unique invariant K\"ahler Einstein metric $\bar g$ with $Ric(\bar g) = \bar g$. It is also known that $\lambda_1(\bar g) = 2$ and that the corresponding eigenspace in $C^\infty(M)$ can be described in terms of the Lie algebra $\frak g$ of $G$. If we now consider the set $\mathcal K_o$ of all $G$-invariant K\"ahler metrics $g'$ on $M$ with $vol(M,g') = vol(M,\bar g)$, we can ask the question when $\bar g$ is an extremal metric for the functional
$\lambda_1$ on $\mathcal K_o$. \par
In our first main result, Theorem \ref{main1}, we will give a sufficient and necessary condition for the metric $\bar g$ to be extremal in terms of invariant data depending only on the complex structure $J$ of $M$ and the Lie algebra $\frak g$, which is supposed to be simple and to coincide with the algebra of the full isometry group of $\bar g$ (this last condition can be assumed without loss of generality). We will also see that there exist flag manifolds whose K\"ahler Einstein metric is extremal.\par
Our second result focuses on the special cases of full flag manifolds, namely spaces of the form $G/T$, where $G$ is a compact simple Lie group and $T$ is maximal torus of $G$. These spaces admit precisely one invariant complex structure up to diffeomorphism and can also be described as the quotient space $G^\mathbb C/B$, where $B$ is a Borel subgroup of the complexification $G^{\small{ \mathbb C}}$ of $G$. In this case we prove in Theorem \ref{fullflag} that the K\"ahler Einstein metric is extremal for $\lambda_1$ on $\mathcal K_o$ if and only if $G = \SU(3)$; moreover in this case we also show that the $\lambda_1|\mathcal K_o$ attains its maximum precisely at $\bar g$.\par
In Section \ref{prel} we recall some standard facts about flag manifolds together with  the definition and main properties of extremality for $\lambda_1$ and in Section \ref{mainsection} we prove our main results. \par
\medskip
\noindent{\bf Notation.} For a compact Lie group, we denote its Lie algebra by the corresponding lowercase gothic letter. If a group $G$ acts on a manifold $M$, for every $X\in \gg$ we denote by $X^*$ the corresponding vector field on $M$ induced by the $G$-action.

\section{Preliminaries}\label{prel}

Given a compact manifold $M$ together with the set $\mathcal R_o$ of all Riemannian metrics on $M$ with fixed volume, the functional $\lambda_1:\mathcal R_o\to \mathbb R$ which assigns to each metric $g\in \mathcal R_o$ the first eigenvalue of the Laplacian $\Delta_g$ acting on $C^\infty(M)$  is continuous (see \cite{Be}). If $g_t$ is an {\it analytic\/} curve of metrics the function $\lambda_1(g_t)$ is not differentiable but still has left and right derivatives in $t$. Indeed given a metric $\bar g$  whose first eigenvalue $\lambda_1(\bar g)$ has multiplicity $m$ and given any analytic curve $g_t$ with $t\in (-\varepsilon,\varepsilon)$ and $g_0=\bar g$, it is know (see \cite{Be}) that there exist $L^2(g_t)$-orthonormal functions $u^1_t,\ldots,u^m_t\in C^\infty(M)$ and $\Lambda^1_t,\ldots,\Lambda^m_t\in \mathbb R$ depending analytically on $t$ such that
for every  $j=1,\ldots,m$ we have $\Lambda^j_0 = \lambda_1(\bar g)$ and
$$\Delta_{g_t}u^j_t = \Lambda^j_t\cdot u^j_t\qquad \forall\ t\in (-\varepsilon,\varepsilon).$$
Then for $t$ small enough, we have $\lambda_1(g_t) = \min_{1\leq i\leq m}\{ \Lambda_t^i\}$ and
$$\frac{d}{dt}|_{t=0^+}\lambda_1(g_t) = \min_{1\leq i\leq m}\{\frac{d}{dt}|_{t=0}\ \Lambda_t^i\},\qquad
\frac{d}{dt}|_{t=0^-}\lambda_1(g_t) = \max_{1\leq i\leq m}\{\frac{d}{dt}|_{t=0}\ \Lambda_t^i\}.$$
On the other hand Berger (\cite{Be}) computed the derivative
$$\frac{d}{dt}|_{t=0}\Lambda^i_t = - \int_M \langle q(u_i),h\rangle\ d\mu_{\bar g},$$
where $h$ is the symmetric tensor given by $\frac{d}{dt}|_{t=0}g_t$, $u_i := u^i_0$ and for every $u\in C^\infty(M)$
\beq\label{quad}q(u) = du\otimes du + \frac 1 4 \ \Delta_{\bar g}u^2\cdot \bar g.\eeq
In \cite{SI} the following definition of extremality has been given
\begin{definition} A metric $\bar g$ is said to be $\lambda_1$-extremal if for every analytic deformation $g_t$ of $\bar g$ in $\mathcal R_o$ we have
$$\frac{d}{dt}|_{t=0^+}\lambda_1(g_t) \leq 0 \leq \frac{d}{dt}|_{t=0^-}\lambda_1(g_t).$$
\end{definition}
Therefore we have the following expressions
\begin{equation}\label{der}\begin{split}
\frac{d}{dt}|_{t=0^+}\lambda_1(g_t) &= -\max_{1\leq i\leq m}\{\int_M \langle q(u_i),h\rangle\ d\mu_{\bar g}\},\\
\frac{d}{dt}|_{t=0^-}\lambda_1(g_t) &= -\min_{1\leq i\leq m}\{\int_M \langle q(u_i),h\rangle\ d\mu_{\bar g}\}.\end{split}
\end{equation}
We now focus on the case where the manifold $M$ is a  compact homogeneous K\"ahler manifold. We consider a compact connected semisimple Lie group $G$ and a compact subgroup $H$ which coincides with the centralizer in $G$ of a torus. The homogeneous space $M= G/H$ is a generalized flag manifold and it can be equipped with invariant K\"ahler structures. We will now state some of the main properties of generalized flag manifolds, referring to \cite{Al,BFR} for a more detailed exposition.\par
We fix a maximal abelian subalgebra $\gt\subseteq\gh$ and the $B$-orthogonal decomposition $\gg = \gh \oplus \gm$, where $B$ denotes the Cartan-Killing form of $\gg$. The subspace $\gm$ can be naturally identified with the tangent space $T_oM$ where $o:=[H]\in G/H$. If $R$ denotes the root system of $\gg^\bC$ relative to the Cartan subalgebra $\gt^\bC$, for every root $\alpha\in R$ the corresponding root space is given by $\gg_\alpha =\bC\cdot E_\alpha$ and
$$\gh^\bC =\gt^\bC\oplus\bigoplus_{\alpha\in R_\gh} \gg_\alpha, \qquad \gm^\bC = \bigoplus_{\alpha\in R_\gm}\gg_\alpha,$$
where $R_\gh\subset R$ is a closed subsystem of roots and $R_\gm := R\setminus R_\gh$. The roots in $R_\gh$ are characterized by the fact that they vanish
on the center $\gc\subseteq \gt$ of $\gh$. Observe that $(R_\gh +R_\gm)\cap R \subseteq R_\gm$. \par
Any $G$-invariant complex structure $J$ on $M$ induces an endomorphism $J\in {\mathrm \End}(\gm)$ with $J^2=-Id$. If we extend $J$ to $\gm^\bC$ and we decompose
$\gm^\bC = \gm^{1,0} \oplus \gm^{0,1}$ into the sum of the $\pm i $ - eigenspaces of $J$, then the integrability of $J$ is equivalent  to the fact that
$\gq:=\gh^\bC \oplus \gm^{1,0}$ is a subalgebra, actually a parabolic subalgebra of $\gg^\bC$. Moreover it can be shown that $G$-invariant complex structures are in bijective correspondence with the invariant orderings of $R_\gm$, namely subsets $R_\gm^+\subset R_\gm$ such that $R_\gm$ is the disjoint union
 $R_\gm = R_\gm^+\ \dot\cup\ (-R_\gm^+)$ and :
$$ (R_\gh + R_\gm^+)\cap R \subset R_\gm^+\ ,\quad (R_\gm^+ +R_\gm^+)\cap R\subset R_\gm^+\ ,$$
the correspondence being given by $\gm^{1,0}= \bigoplus_{\alpha\in R_\gm^+}\gg_\alpha$. Invariant orderings are then in one-to-one correspondence with Weyl chambers in the center $\gc$ of $\gh$, namely connected components of the set $\gc\setminus \bigcup_{\alpha\in R_\gm}\ker (\alpha|_\gc)$, and an invariant ordering in $R_\gm$ can be combined with an ordering in $R_\gh$ to provide a standard ordering in $R$.\par
If we fix an invariant complex structure $J$ on $M$ (hence a Weyl chamber $C$ in $\gc$), we can endow $M$ with many $G$-invariant K\"ahler metrics which are Hermitian w.r.t. $J$. Actually, it can be proved that $G$-invariant symplectic structures, namely $G$-invariant non-degenerate closed two-forms, are in one-to-one correspondence with elements in the
Weyl chambers in $\gc$. Indeed, if $\omega\in \Lambda^2(\gm)$ is a symplectic form, then there exists $\xi$ in some Weyl chamber in $\gc$ such that
\beq\label{forms}\omega(X,Y) = B(\ad_\xi X,Y),\quad X,Y\in \gm.\eeq
Moreover $\omega$ is the K\"ahler form of a K\"ahler metric $g$ w.r.t. the complex structure $J$ (i.e. $g:= \omega(\cdot,J\cdot)$ defines a K\"ahler metric) if and only if $\xi\in C$. \par
The functional
\beq\label{delta} \d_\gm := \sum_{\a\in R_\gm^+}\a\eeq
plays a very important role. Indeed, its dual $\hat\d_\gm = \sum_{\a\in R_\gm^+}H_\a\in \gt^\bC$, where $H_\a$ denotes the $B$-dual of the root $\a$,  lies in $i\gc$ while the element $\eta_\gm := -i\hat\d_\gm$ belongs to $C$ and it represents (via the correspondence \eqref{forms}) the Ricci form of every invariant metric which is K\"ahler w.r.t. the invariant complex structure $J$ (see e.g.~\cite{BFR},~ p. 627). It then follows that the element $\eta_\gm$ itself defines via \eqref{forms} an invariant K\"ahler metric which is the unique invariant K\"ahler Einstein metric $\bar g$ with $Ric(\bar g) = \bar g$. It is a well known fact (see e.g.~\cite{Ko},~p.96) that $\lambda_1(\bar g) = 2$ and that the relative eigenspace $E$ is isomorphic to the Lie algebra of all Killing fields on $(M,\bar g)$ via the isomorphism
\beq E_1\ni f\mapsto J\mbox{grad}(f).\eeq
 \par
\bigskip
\section{The main results}\label{mainsection}
\medskip
\medskip
Keeping the same notations as in the previous section, we consider a compact homogeneous space $M = G/H$, where $G$ is a semisimple compact Lie group and $H$ is the centralizer in $G$ of a torus, endowed with an invariant complex structure $J$. The set of all K\"ahler metrics (w.r.t. the complex structure $J$) is then
parametrized by the points in the open Weyl chamber $C$ in the center $\gc$ of $\gh$ corresponding to the complex structure $J$. If we now consider the hypersurface $\mathcal K_o\subset C$ given by those points in $C$ which corresponds to invariant K\"ahler metrics with the same volume as the K\"ahler Einstein metric $\bar g$, we are interested in the functional $\lambda_1^{{IK}}:\mathcal K_o\to \bR$. In particular we would like to study the question when the metric $\bar g$ is $\lambda_1^{IK}$-extremal, namely
\begin{definition}\label{KI} The K\"ahler Einstein metric $\bar g$ is said to be {\it $\lambda_1^{IK}$-extremal\/} if for every analytic curve $g_t$ in $\mathcal K_o$, $t\in (-\epsilon,\epsilon)$, with $g_0 = \bar g$, we have
$$\frac{d}{dt}|_{t=0^+}\lambda_1(g_t) \leq 0 \leq \frac{d}{dt}|_{t=0^-}\lambda_1(g_t).$$
\end{definition}
Given any analytic variation $g_t$ as in the definition, we put $h:= \frac{d}{dt}|_{t=0}g_t$. The symmetric tensor $h$ is also $G$-invariant and therefore the scalar product $\langle h,\bar g\rangle$ is a constant on $M$. Moreover, since $vol(g_t) = vol(\bar g)$ we see that $\int_M\ \langle h,\bar g\rangle\ d\mu_{\bar g} = 0$, hence $\langle h,\bar g\rangle = 0$. It is clear that the set of all symmetric tensors which are tangent vectors to variations $g_t$ can be identified with $\frak c$ and those which correspond to variations of constant volume with the hyperplane $Y := \ker \mbox{Tr}\subset \frak c$, where $\mbox{Tr}(h) = \langle h,\bar g\rangle$.\par
 If we now consider the quadratic form $Q_h$ on the space $E$ (see \eqref{quad}) given by
$$Q_h(u) = \int_M \langle du\otimes du + \frac 14 \Delta_{\bar g}(u^2)\cdot \bar g,h\rangle\ d\mu_{\bar g},$$
we see that it can be simplified to
\beq\label{Q_h}Q_h(u) = \int_M \langle du\otimes du ,h\rangle\ d\mu_{\bar g}.\eeq
We now recall that the space $E$ is isomorphic to the Lie algebra of Killing vector fields on $(M,\bar g)$.\par
Now if $G=_{\mathrm{loc}}G_1\times\ldots\times G_k$ is the decomposition of $G$ into a product of simple factors, then $H$ splits accordingly as $H=_{\mathrm{loc}}H_1\times\ldots\times H_k$ for $H_i\subset G_i$ and $M$ is biholomorphically isometric to the
product of irreducible  homogeneous spaces $M=M_1\times\ldots\times M_k$, $M_i:= G_i/H_i$, endowed with invariant complex structures and K\"ahler Einstein metrics $\bar g_i$ for $i=1,\ldots,k$. The Lie algebra of holomorphic automorphisms $\frak{aut}(M,J)$ coincides with the complexification of the algebra $\imath(M,\bar g)$ of isometries of $(M,\bar g)$ and it splits as $\bigoplus_{i=1}^k\frak{aut}(M_i,J_i)$, where again $\frak{aut}(M_i,J_i) = \imath(M_i,\bar g_i)^\bC$. The inclusion $\gg_i\subseteq \imath(M_i,\bar g_i)$ is proper in a few cases which are classified by Onishchik (see~\cite{On}) and also in these cases the algebra
of infinitesimal isometries is simple. Therefore without loss of generality we can suppose that each factor $\gg_i$ coincides with the algebra $\imath(M_i,\bar g_i)$.\par
We first deal with the case where $G$ is simple. We note that on $E$  the $G$-invariant  $L^2(\bar g)$ inner product is given by $\kappa\cdot B$ for a suitable non-zero constant $\kappa$, so that the $L^2$-orthonormal basis $u_1,\ldots, u_m$ of $E$ is also $B$-orthogonal. For $h\in\gc$ the quadratic form $Q_h$ on $E$ given by
$$Q_h(u) := \int_{M}\ \langle du\otimes du,h\rangle\ d\mu_{\bar g}$$
is also $G$-invariant and therefore there exists $\phi\in \frak{c}^*$ such that
\beq \label{structure} Q_h = \phi(h)\cdot B.\eeq
It then follows that
$$\max_{1\leq i\leq m} Q_{h}(u_i,u_i) = \min_{1\leq i\leq m} Q_{h}(u_i,u_i) = \frac 1{\kappa} \phi(h). $$
By the definition~\eqref{KI} and~\eqref{der} we get the following
\begin{lemma} Let $G$ be simple and suppose that $G$ coincides (locally) with the group of isometries of $(M,\bar g)$. Then the K\"ahler Einstein metric is $\lambda_1^{KI}$-extremal if and only if the functional $\phi$ vanishes identically on the hyperplane $Y$, i.e. if and only if there exists a constant $c\in \bR$ so that
\beq\label{eq1} \phi = c\cdot \mbox{Tr}.\eeq
\end{lemma}
We now elaborate~\eqref{eq1} in terms of algebraic data in $\gg$. We consider the standard Weyl basis $\{E_\alpha\}_{\alpha \in R}$ of the root spaces $\gg_\alpha$ (see e.g.~\cite{He}, p.~421) and we construct a $\bar g$-orthonormal basis of $\gm$ by considering for $\alpha \in R_{\gm}^+$
$$v_\alpha := \frac 1{\sqrt{2\ B(\alpha,\delta_\gm)}}\ (E_\alpha - E_{-\alpha}),\quad w_\alpha = Jv_\alpha = \frac i{\sqrt{2\ B(\alpha,\delta_\gm)}}\ (E_\alpha + E_{-\alpha}). $$
Now if $h$ is a $G$-invariant symmetric Hermitian $(0,2)$-tensor given by $h(v,v) = B([\xi_h,v],Jv)$ for every $v\in \gm$ and for some $\xi_h \in \frak c$, we have that
\beq\label{Tr}{\mbox{Tr}}(h) = 2\ \sum_{\alpha\in R_\gm^+} h(v_\alpha,v_\alpha) = 2i\  \sum_{\alpha\in R_\gm^+} \frac{\alpha(\xi_h)}{B(\alpha,\delta_\gm)}\ .\eeq
In order to compute the functional $\phi$ we consider a $(-B)$-orthonormal basis $\{v_j\}_{1\leq j\leq N}$ of $E$, where we recall that the map $T:E\to \gg$ given by $T(v) = J\mbox{grad}v$ is an isomorphism. We then consider
$$\sum_{j=1}^N Q_h(v_j) = - N\cdot \phi(h),$$
where $N=\dim \gg$. We now observe that the symmetric bilinear form $b:= \sum_j dv_j\otimes dv_j$ is $G$-invariant and therefore the function $\langle b,h\rangle$ is constant. Hence we get
$$- N\cdot \phi(h) = vol(M,\bar g) \cdot \langle b,h\rangle. $$
Now in order to compute the scalar product $\langle b,h\rangle$, we fix a $\bar g$-orthonormal basis $e_1,\ldots,e_n$ ($n=\dim_{\small \bR}M$) at $T_{[H]}M$ and compute
$$\langle b,h\rangle = \sum_{i,j}(dv_j(e_i))^2h(e_i,e_i) = \sum_{ij} h(\bar g(e_i,\mbox{grad}v_j)e_i,\bar g(e_i,\mbox{grad}v_j)e_i) = $$
$$= \sum_j h(J\mbox{grad}v_j,J\mbox{grad}v_j) = \sum_j h(T(v_j),T(v_j)),$$
so that $\langle b,h\rangle$ is simply given by the trace $\mbox{Tr}_{-B}(h)$ of $h\in S^2(\gm^*)$ w.r.t. the inner product $-B$ on $\gm$.\par
Hence
$$\langle b,h\rangle = \sum_{\alpha\in R_\gm^+} h(E_\alpha - E_{-\alpha}, E_\alpha - E_{-\alpha}) = $$
$$= \sum_{\alpha\in R_\gm^+} i\ B([\xi_h,E_\alpha - E_{-\alpha}],E_\alpha + E_{-\alpha}) = 2i\ \sum_{\alpha\in R_\gm^+}\alpha(\xi_h).$$
Now condition~\eqref{eq1} can be written as
$$\sum_{\alpha\in R_\gm^+}\alpha = \mu\cdot \sum_{\alpha\in R_\gm^+} \frac{\alpha}{B(H_\alpha,\delta_\gm)}$$
for some real constant $\mu$. By contracting both members with $\delta_\gm$, we obtain
$$||\delta_\gm||^2 = \mu\cdot \dim_\bC M.$$
Therefore we have proved the following
\begin{theorem}\label{main1} Let $G$ be a compact and simple Lie group and let $M = G/H$ be a flag manifold endowed with a $G$-invariant complex structure $J$. Let $\bar g$ be the K\"ahler Einstein with $Ric_{\bar g} = \bar g$ and suppose that $G$ coincides (locally) with the full isometry group of $\bar g$. \par
The metric $\bar g$ is $\lambda_1^{KI}$-extremal if and only if
 \beq\label{condition}\delta_\gm|_{\frak c} = \frac{||\delta_\gm||^2}{\dim_{\tiny\bC} M}\cdot \sum_{\alpha\in R_\gm^+} \frac{\alpha|_{\frak c}}{B(\alpha,\delta_\gm)},\eeq
where $\frak c$ is the center of the Lie algebra $\gh$ of $H$.
\end{theorem}
\begin{remark} {\rm The condition of $\l_1^{KI}$-extremality makes sense when $\dim \gc \geq 2$, hence when $b_2(M)\geq 2$.}\end{remark}
\begin{remark} {\rm Equation~\eqref{condition} provides a simple and computable condition to be tested in a flag manifold $M$. Actually it is an equation
involving only the $T$ roots as we now explain. If $p:(i\frak t)^*\to (i\frak c)^*$ denotes the restriction map, a $T$-root is by definition the image $p(\a)$ for some root $\a\in R$. The set $R_T := p(R)$ of $T$-roots is not a root system and  it is in one-to-one correspondence with the set of irreducible $H$-submodules $\gm_j$ of $\gm^\bC$, $j=1,\ldots,t$, where each $\rho\in R_T$ corresponds to the submodule $\sum_{\alpha\in R_\gm,\ p(\a)=\rho}\gg_\a$ (see e.g. \cite{AP,S}). We denote by $R_T^+ = p(R^+) = p(R_\gm^+)$ the set of  $T$-roots which are the image of positive roots, say $R_T^+ = \{\rho_1,\ldots,\rho_\ell\}$, where each $\rho_j$ corresponds to an irreducible submodule $\gm_j$ of complex dimension $m_j$, $j=1,\ldots,\ell$. Then $t=2\ell$ and $\gm^\bC = \bigoplus_{i=1}^\ell \gm_j \oplus\bar\gm_j$. Moreover, for each $\rho\in R_T^+$, the quantity $B(\alpha,\d_\gm)$ with $p(\a)=\rho$ does not depend on $\alpha$ with $p(\a) = \rho$ and it is denoted by $B(\rho,\d_\gm$). Then condition~\eqref{condition} can be rewritten as
\beq\label{conditionbis}\sum_{j=1}^\ell \left(\frac \mu{\b_j} - 1\right) m_j\ \rho_j = 0, \eeq
where $\b_j:= B(\rho_j,\d_\gm)$ and $\mu = ||\d_\gm||^2/\dim_\bC M$. }\end{remark}

\bigskip\par\bigskip
\begin{example} {\rm We consider the flag manifold $M = \SU(3n)/{\rm{S}}(\U(n)\times\U(n)\times\U(n))$, $n\geq 1$, endowed with the complex structure $J$ corresponding to the standard positive root system $R^+ = \{\e_i-\e_j|\ 1\leq i<j\leq 3n\}$ of $\su(3n)$. We have that $R_\gm^+ = \{\e_i-\e_j|\ 1\leq i\leq n,\
n+1\leq j\leq 3n\}\cup \{\e_i-\e_j|\ n+1\leq i\leq 2n,\ 2n+1\leq j\leq 3n\}$ and $\d_\gm = 2n(\sum_{i=1}^n \e_i - \sum_{i=2n+1}^{3n}\e_i)$. There are precisely three positive $T$-roots and six irreducible $H$-irreducible submodules of $\gm^\bC$ of complex dimension $n^2$. We have $\rho_1 = \e_1-\e_{n+1}$, $\rho_2 = \e_1-\e_{2n+1}$ and $\rho_3 = \e_{n+1}-\e_{2n+1}$. If we normalize the Cartan Killing form of $\su(3n)$ in such a way that $||\a||^2= 2$ for every root $\a$, then $\b_1 = \b_3 = 2n$ and $\b_2=4n$ and $\mu = \frac 83 n$. Then condition \ref{conditionbis} reads
$$\frac 13 (\rho_1-\rho_2+\rho_3) = 0,$$
so that Theorem \ref{main1} applies and the K\"ahler Einstein metric on ($M,J$) is $\lambda_1^{KI}$-extremal.\par
On the other hand we will see in Theorem \ref{fullflag} that there are many flag manifolds whose K\"ahler Einstein metric is not $\l_1^{KI}$-extremal.}
\end{example}

\bigskip
We now turn to the reducible case, namely when $M = M_1\times\ldots\times M_k$, where $M_j= G_j/H_j$ and $G_J$ are compact simple Lie groups for $j=1,\ldots,k$. The invariant complex structure $J$ is the product of $G_j$-invariant complex structures $J_j$ on $M_j$ as well as the K\"ahler Einstein metric $\bar g$ which is isometric to the product metric $\bar g_1\times \ldots\times \bar g_k$. We prove the following
\begin{theorem} The K\"ahler Einstein metric $\bar g$ on $M=M_1\times\ldots\times M_k$ is $\lambda_1^{KI}$-extremal if and only each $\bar g_i$ is $\lambda_1^{KI}$-extremal for $i=1,\ldots,k$.
\end{theorem}
\begin{proof}

If we consider an analytic curve of $G$-invariant K\"ahler metric $g_t$ with $g_0 = \bar g$, then it splits as a product of invariant metrics $g^{(i)}_t$ on $M_i$  and $h:= \frac{d}{dt}|_{t=0}g_t$ also splits as $h = h_1\times \ldots\times h_k$ with $h_i$ being $J_i$-Hermitian $G_i$-invariant symmetric tensors on $M_i$ for $i=1\ldots k$. If $vol(M,g_t)$ is constant, then $\sum_i \mbox{Tr}_{\bar g_i}(h_i) = 0$. The space $E\cong \gg$ splits accordingly as $E= \bigoplus_{i=1}^kE_i$ with $E_i\cong \gg_i$ and $\lambda_1(g_t) = \min_{1\leq j\leq k}\{\lambda_1(g^{(j)}_t)\}$. \par
It is clear that choosing a variation $g_t$ in only one factor $M_i$, the $\lambda_1^{KI}$-extremality of $\bar g$ implies the one of $\bar g_i$ for each $i=1,\ldots k$.\par
Viceversa given an analytic deformation of $\bar g$ as above, we consider the quadratic forms $Q^{(i)}_{h_i}$ on $E_i$ given by
$$Q^{(i)}_{h_i}(u) = \int_{M_i}\ \langle du\otimes du,h_i\rangle\ d\mu_{\bar g_i} = \phi_i(h_i)\cdot B_i$$
for some $\phi_i\in \frak c_i^*$, where $B_i$ denotes the Cartan-Killing form on $E_i\cong \gg_i$.
We also consider the $L^2(\bar g_i)$-orthonormal sets $\{u^{(i)}_1,\ldots,u^{(i)}_{N_i}\}\subset E_i$ for $i=1,\ldots,k$ and $N_i = \dim \gg_i$ corresponding to the variations $g^{(i)}_t$ as explained in \S 2. Then by Theorem \ref{main1} we have $\phi_i(h_i) = c_i\cdot \mbox{Tr}(h_i)$ and therefore
$$\frac d{dt}|_{t=0^+}\lambda_1(g_t) = -\max_{1\leq i\leq k}\{Q^{(i)}_{h_i}(u^{(i)}_{1}), \ldots,Q^{(i)}_{h_i}(u^{(i)}_{N_i})\} = $$
$$ = -\max_{1\leq i\leq k}\{ \frac {c_i}{\kappa_i} \mbox{Tr}(h_i)\},$$
where again we denote by $\kappa_i$ the negative constant such that $\kappa_i\cdot B_i$ is equal to the $L^2(\bar g_i)$-inner product on $E_i$. Now, we note that $\kappa_i<0$ and $c_i>0$ for all $i$. Since $\sum_i \mbox{Tr}(h_i) = 0$, there exists $i$ so that $\mbox{Tr}(h_i)\leq 0$ and therefore $\frac d{dt}|_{t=0^+}\lambda_1(g_t) \leq 0$. Similarly we get $\frac d{dt}|_{t=0^-}\lambda_1(g_t)\geq 0$, hence $\bar g$ is $\lambda_1^{KI}$-extremal.\end{proof}
\medskip
\subsection{The case of full flag manifolds}
In this subsection we focus on the case of a full flag manifold, namely a homogeneous space $M= G/T$ where $G$ is compact simple, $T$ is a maximal torus in $G$ and $\dim T\geq 2$. It is known (see e.g.~\cite{BH}) that two invariant complex structures on $M$ are biholomorphic and therefore we fix one invariant complex structure $J$ with corresponding Weyl chamber $C$ in $\gt$. The main result is the following
\begin{theorem}\label{fullflag} If $G$ is a compact simple classical group, the $G$-invariant K\"ahler Einstein metric on $(G/T,J)$ is $\lambda_1^{KI}$-extremal if and only $G = \SU(3)$. \par
Moreover when $M=\SU(3)/T^2$, the eigenvalue functional $\lambda_1:\mathcal K_o\to \bR$ attains its maximum at $\bar g$.
\end{theorem}
\begin{proof} In order to prove the first assertion, we will go through the classical simple Lie algebras checking condition~\eqref{condition}. In particular we will use the general fact that the element $\delta = \sum_{\a\in R^+}\a$ satisfies $<\delta,\alpha> = ||\alpha||^2$ whenever $\alpha$ is a simple root (see e.g. ~\cite{He}). \par
When $\gg = \su(n)$, we consider the standard set of roots $R= \{\e_{i,j}:=\epsilon_i-\epsilon_j|\ 1\leq i\neq j\leq n\ \}$ where the positive roots are given
by $R^+ = \{\e_{i,j}|\ i < j\ \}$. We may also normalize the Cartan Killing form by setting $||\alpha||^2=2$ for every root $\alpha$. It is immediate to compute $\d = \sum_{j=1}^n (n+1-2j)\e_j$, so that
$$||\delta||^2 = \frac 13 n(n-1)(n+1).$$
If condition~\eqref{condition} holds and $\bar\alpha$ is a simple root, then
\beq\label{su}2 = \frac{||\delta||^2 }{\dim_{\small{\bC}}M}\sum_{\alpha\in R^+}\frac{\langle \alpha,\bar\alpha\rangle}{\langle \alpha,\delta\rangle} =
\frac 23 (n+1) \sum_{\alpha\in R^+}\frac{\langle \alpha,\bar\alpha\rangle}{\langle \alpha,\delta\rangle}\ .\eeq
If we select $\bar \alpha = \e_{1,2}$ then the set $A:=\{\a\in R^+;\ \langle\a,\bar \a\rangle \neq 0\}$ is given by
$A = \{\e_{1,2},\ldots,\e_{1,n},\e_{2,3},\ldots,\e_{2,n}\}$ so that
$$\sum_{\alpha\in R^+}\frac{\langle \alpha,\bar\alpha\rangle}{\langle \alpha,\delta\rangle} = 1 + \frac 12\cdot \sum_{j=2}^{n-1}\frac 1{j} - \frac 12 \cdot\sum_{j=1}^{n-2}\frac 1j = \frac n{2(n-1)}, $$
so that condition \eqref{su} reads $6(n-1)=n(n+1)$, i.e. $n=2,3$.  On the other hand if $n=3$ then $\delta = 2\e_{1,3}$ and $\sum_{\a\in R^+}\frac{\a}{\langle \a,\d\rangle} = \frac 34 \e_{1,3}$, so that condition~\eqref{condition} holds true.\par
When $\gg = \so(2n+1)$ ($n\geq 2$) we consider the set of simple roots $\{\o_i-\o_{i+1}, \o_n\ i=1,\ldots,n-1 \}$ and the set $R^+=\{ \o_i\pm\o_j,\ 1\leq i<j\leq n\}\cup\{\o_i\ 1\leq i\leq n\}$ for an orthonormal basis $\{\o_i\}$ of $\gt$. We have $\d = \sum_{i=1}^n (2n-2i +1)\o_i$ and therefore
$||\d||^2 = \frac 13 n(4n^2-1)$. If we consider the simple root $\bar\a=\o_n$, the set of positive roots $A=\{\a\in R^+|\ \langle \a,\bar\a\rangle \neq 0\}$ is given by $A= \{\o_i\pm \o_n;\ i<n\}\cup\{ \o_n\}$ and therefore condition \eqref{condition} implies
$$\frac{\dim_{\small{\bC}} M} {||\d||^2}= \sum_{\alpha\in R^+}\frac{\langle \alpha,\bar\alpha\rangle}{\langle \alpha,\delta\rangle} = \sum_{i=1}^{n-1}\frac 1{2(n-i+1)} - \sum_{i=1}^{n-1} \frac 1{2(n-i)} + 1 = \frac 12\left(1+\frac 1n\right).$$
Since $\dim_{\small{\bC}}M = n^2$, we have $6n^2=(4n^2-1)(1+n)$, which has no integer solution $n\geq 2$.\par

When $\gg = \sp(n)$ ($n\geq 3$), we fix a standard system of simple roots $\{\o_1-\o_2,\ldots,\o_{n-1}-\o_n, 2\o_n\}$ with system of positive roots given by
$R^+= \{\o_i\pm\o_j, 2\o_i;\ 1\leq i< j\leq n\}$. Then $\d = 2\sum_{i=1}^n (n+1-i)\o_i$ and $||\d||^2 = \frac 23 n(n+1)(2n+1)$. If we choose the simple root $\a=2\o_n$ the set of positive roots $A=\{\a\in R^+|\ \langle \a,\bar\a\rangle \neq 0\}$ is given by $A= \{\o_i\pm \o_n;\ i<n\}\cup\{ 2\o_n\}$ and therefore condition \eqref{condition} implies
$$4\frac{\dim_{\small{\bC}} M} {||\d||^2}= \sum_{\alpha\in R^+}\frac{\langle \alpha,\bar\alpha\rangle}{\langle \alpha,\delta\rangle} = \sum_{i=1}^{n-1}\frac 1{n+2-i} - \sum_{i=1}^{n-1} \frac 1{n-i} + 1 = \frac 1{n+1} + \frac 1n - \frac 12.$$
Noting that  $\dim_{\small{\bC}}M = n^2$ and that the right hand side of the above equation is positive only for $n=3$, we immediately see that we have no integer solution.\par
When $\gg = \so(2n)$ ($n\geq 3$), we fix a standard system of simple roots $\{\o_1-\o_2,\ldots,\o_{n-1}-\o_n, \o_{n-1}+\o_n\}$ with system of positive roots given by $R^+= \{\o_i\pm\o_j;\ 1\leq i< j\leq n\}$. Then $\d = 2\sum_{i=1}^n (n-i)\o_i$ and $||\d||^2 = \frac 23 n(n-1)(2n-1)$. If we choose the simple root $\a=\o_1-\o_2$ the set of positive roots $A=\{\a\in R^+|\ \langle \a,\bar\a\rangle \neq 0\}$ is given by $A= \{\o_1-\o_2, \o_1\pm \o_i, \o_2\pm\o_i;\ 3\leq i\leq n\}$ and therefore condition \eqref{condition} implies
$$2\frac{\dim_{\small{\bC}} M} {||\d||^2}= \sum_{\alpha\in R^+}\frac{\langle \alpha,\bar\alpha\rangle}{\langle \alpha,\delta\rangle} = 1 +\frac 12 \left(\sum_{i=3}^{n}\frac 1{2n-1-i} + \frac 1{i-1} \right)- \frac 12 \left(\sum_{i=3}^{n} \frac 1{2n-2-i} + \frac 1{i-2}\right)=$$
$$= \frac 1{n-1} - \frac 1{4n-10} + \frac 12.$$
Since $\dim_{\small{\bC}} M = n^2-n$, the previous equation is equivalent to $3(n-1)(2n-5) = (2n-1)(n^2-n-3)$, which has no solution for $n\geq 3$.\par

In order to prove the last assertion regarding $M = \SU(3)/T^2$, we fix some notations. We consider the abelian subalgebra $\gt = \{\mbox{diag}(ia,ib,-i(a+b));\ a,b\in \bR\}\cong \bR^2$ and a standard basis of the complement $\gm = \gt^\perp$ w.r.t. the Cartan Killing form $B$, which is given by $B(X,Y) = 6\ \mbox{Tr}(XY)$:
\begin{equation}\label{basis} \begin{split} X_1 &= {\left(\begin{smallmatrix} 0&1&0\\ -1&0&0\\0&0&0\end{smallmatrix}\right)},\quad X_2 = {\left(\begin{smallmatrix} 0&0&1\\ 0&0&0\\-1&0&0\end{smallmatrix}\right)}, \quad X_3 = {\left(\begin{smallmatrix} 0&0&0\\ 0&0&1\\0&-1&0\end{smallmatrix}\right)} \\
Y_1 &= \left(\begin{smallmatrix} 0&i&0\\ i&0&0\\0&0&0\end{smallmatrix}\right),\quad Y_2 = \left(\begin{smallmatrix} 0&0&i\\ 0&0&0\\i&0&0\end{smallmatrix}\right)
, \quad Y_3 = \left(\begin{smallmatrix} 0&0&0\\ 0&0&i\\0&i&0\end{smallmatrix}\right).\end{split}\end{equation}
Given the set of simple roots $\{\e_{1,2},\e_{2,3}\}$, we fix the Weyl chamber $C = \{(a,b)\in \bR^2\cong\gt;\ -\frac a2<b<a\}$. It corresponds to the invariant complex structure $J$ such that $JX_i=Y_i$, $i=1,2,3$. The element $\d = 2\ \e_{1,3}$ has a $B$-dual given by $\bar\xi:=\frac 13\ {\mbox{diag}}(i,0,-i)$, i.e. the point $(\frac 13,0)\in \bR^2$.\par
Any $\xi\in C$ determines an invariant K\"ahler metric $g_\xi$ according to \eqref{forms}. The $\mbox{Ad}(T^2)$-submodules $\gm_i:=\mbox{Span}\{X_i,Y_i\}$ are mutually inequivalent and therefore they are $g_\xi$-orthogonal. Moreover we easily get $g_\xi(X_i,X_i)=g_\xi(Y_i,Y_i)$ ($i=1,2,3$) and
$$g_\xi(X_1,X_1)= 12(a-b);\ g_\xi(X_2,X_2) = 12(2a+b);\ g_\xi(X_3,X_3) = 12(a+2b).$$
If we now put $s:= a-b$ and $t:= a+2b$, then the set $\mathcal K_o$ of invariant K\"ahler metrics whose volume is equal to $vol(M,\bar g)$ is given by
$$\mathcal K_o \cong \{(s,t)|\ s,t > 0,\ st(s+t) = \frac 2{27}\}.$$
In order to estimate the first eigenvalue of $g_\xi$, we recall some well-known facts about the Laplacian of an invariant metric on a homogeneous space. Indeed,
if $\{v_1,\ldots,v_6\}$ is an orthonormal basis of $\gm$ w.r.t. an invariant metric $q$ on $\SU(3)/T^2$ and if $\rho:\SU(3)\to \SU(V)$ is an irreducible representation on a Hermitian vector space $V$ with $V^{T^2} = \{v\in V;\ hv=v\ \forall h\in T^2\} \neq \{0\}$, then the operator
$$D_\rho := \sum_{1=1}^6 \rho(v_i)^2$$
leaves $V^{T^2}$ invariant and its spectrum $\mbox{Spec}(D_\rho,V^{T^2})$ is contained in the spectrum of $\Delta_q$ acting on $C^\infty(M)$ (see e.g.~\cite{MU}).\par
We will consider the irreducible representation $\mbox{Ad}$ of $\SU(3)$ given by the adjoint representation $V=\su(3)^{\small{\bC}}$. In order to compute
$D_{\mbox{Ad}}$ relative to the metric $g_{(s,t)}\in \mathcal K_o$, we fix the $g_{(s,t)}$-orthonormal basis given by
$$v_1 = \frac 1{\sqrt{12 s}} X_1,\ v_2 = \frac 1{\sqrt{12 (s+t)}} X_2,\ v_3 = \frac 1{\sqrt{12 t}} X_3,\ v_i= Jv_{i-3},\ i=4,5,6. $$
The space $V^{T^2}$ coincides with the Lie subalgebra $\gt^{\small{\bC}}$ which can be identified with $\bC^2$ via $(z_1,z_2)\mapsto \mbox{diag}(z_1,z_2,-z_1-z_2)$. A lengthy but straightforward computation shows that the endomorphism $D_{\mbox{Ad}}$ of $\gt^{\small{\bC}}$ can be represented by the matrix
$$D_{\mbox{Ad}} = \frac 13 \cdot \left(\begin{matrix} \frac{3s+t}{s(s+t)}& -\frac t{s(s+t)}\\ \frac{s-t}{st} & \frac{2s+t}{st}\end{matrix}\right)$$
and since we are considering metrics in $\mathcal K_o$, i.e. with $st(s+t) = \frac 2{27}$, it can be rewritten as
$$D_{\mbox{Ad}} = \frac 92 \cdot \left(\begin{matrix} t(3s+t)& -t^2\\ s^2-t^2 & (2s+t)(s+t)\end{matrix}\right).$$
Its eigenvalues are given by $\frac 92 (t^2+s^2+3st\pm\sqrt{t^4+s^4-s^2t^2})$. Note that we get eigenvalues $\{2,3\}$ for $s=t=\frac 13$, which corresponds to the K\"ahler Einstein metric $\bar g$.\par
Using a software, e.g. Maple, it can be shown that the function $f(s,t) =: \frac 92 (t^2+s^2+3st-\sqrt{t^4+s^4-s^2t^2})$ restricted to the curve $\{st(s+t) = \frac 2{27}\}$ attains its maximum value $2$ at $s=t=\frac 13$. This shows in particular that $\lambda_1(g_{(s,t)}) < 2$ for $(s,t)\in \mathcal K_o$, $(s,t)\neq (\frac 13,\frac 13)$ and our last claim is proved. \end{proof}

\bigskip\bigskip

\end{document}